\documentclass[10pt,a4paper,reqno]{amsart}
\usepackage[T1]{fontenc}
\usepackage{amssymb}
\usepackage{amsmath}
\usepackage{amsfonts}
\usepackage{ifthen}
\usepackage{leftidx}
\usepackage{comment}
\usepackage{mathtools}

\newcounter{margnotes}

\def\sideremark#1{\ifvmode\leavevmode\fi\vadjust{\vbox to0pt{\vss 
			\hbox to 0pt{\hskip\hsize\hskip1em           
				\vbox{\hsize3cm\tiny\raggedright\pretolerance10000
					\noindent #1\hfill}\hss}\vbox to8pt{\vfil}\vss}}}%

\newcounter{lemenumi}
\newcommand{\labelemenumi}{(\alph{lemenumi})}

\newtheorem{theorem}{Theorem}[section]
\newtheorem{lemma}[theorem]{Lemma}
\newtheorem{conjecture}[theorem]{Conjecture}
\newtheorem{proposition}[theorem]{Proposition}

\theoremstyle{definition}

\newtheorem{definition}[theorem]{Definition}

\newtheorem{approach}[theorem]{General approach}

\theoremstyle{remark}
\newtheorem{remark}[theorem]{Remark}

\numberwithin{equation}{section}

\newcommand{\OO}{\mathcal{O}}

\newcommand{\Ocom}{\pi_1^{\OO-ab}}
\DeclareMathOperator{\Diff}{Diff}

\begin{document}
	\newcommand{\xc}{\xi}
	\newcommand{\yc}{\eta}
	\newcommand{\ep}{\varepsilon}
	\newcommand{\jb}{II}
	\newcommand{\df}{d}
	\newcommand{\al}{\alpha}
	
	\newcommand{\testrat}{R}
	\newcommand{\testpol}{P}
	
	\newcommand{\pn}{Q}
	\newcommand{\qn}{Q_1}
	\newcommand{\mon}{\mathcal{M}on}
	\newcommand{\ve}{\varepsilon}
	\newcommand{\var}{\mathcal{V}ar}
	\newcommand{\org}{T}
	\newcommand{\orgg}{t}
	\newcommand{\supp}{\text{supp}}
	\newcommand{\R}{\mathbb{R}}
	\newcommand{\C}{\mathbb{C}}
	\newcommand{\I}{\mathbb{I}}
	\newcommand{\N}{\mathbb{N}}
	\newcommand{\Z}{\mathbb{Z}}
	\newcommand{\Q}{\mathbb{Q}}
	\newcommand{\Xbar}{\bold{X}}

	
	\newcommand{\edg}{\gamma^0}
	\newcommand{\vtx}{p}
	\newcommand{\CH}{{\mathbb{C}} H_1({\mathcal{O}})}
	\newcommand{\CHy}{{\mathbb{C}} H_1({\mathcal{O}_y})}
	
	\title[Nilpotence and length of Melnikov functions]{Nilpotence of Orbits under Monodromy\\
	and\\
	the Length of Melnikov Functions}

	\author[P. Marde\v si\'c]{Pavao Marde\v si\'c}
\address[a]{Universit\'e de Bourgogne, Institute de 
	Math\'ematiques de Bourgogne - UMR 5584 CNRS\\
	Universit\'e de Bourgogne Franche Comt\'e
	9 avenue Alain Savary,
	BP 47870, 21078 Dijon, \\FRANCE}
\address[a]{ University of Zagreb, Faculty of Science, Department of Mathematics, Bijeni\v cka 30, 10 000 Zagreb, Croatia}
\email{mardesic@u-bourgogne.fr}

\author[D. Novikov]{Dmitry Novikov}
\address[b]{Faculty of Mathematics and 
	Computer Science, Weizmann Institute of Science, Rehovot, 7610001\\Israel}
	\email{dmitry.novikov@weizmann.ac.il}

\author[L. Ortiz-Bobadilla]{Laura Ortiz-Bobadilla}
\address[c]{Instituto de Matem\'aticas, Universidad Nacional Aut\'onoma de 
M\'exico 	(UNAM), 	\'Area de la Investigaci\'on Cient\'ifica, Circuito 
exterior, Ciudad 	Universitaria, 04510, Ciudad de M\'exico, M\'exico}
\email{laura@matem.unam.mx}

\author[J. Pontigo-Herrera]{Jessie Pontigo-Herrera}
\address[d]{Instituto de Matem\'aticas, Universidad Nacional Aut\'onoma de 
	M\'exico 	(UNAM), 	\'Area de la Investigaci\'on Cient\'ifica, Circuito 
	exterior, Ciudad 	Universitaria, 04510, Ciudad de M\'exico, M\'exico}
\email{pontigo@matem.unam.mx}
	
	
\thanks{This work was supported by Israel Science Foundation grant 1167/17, 
 Papiit (Dgapa UNAM) IN110520, Croatian Science Foundation (HRZZ) grants no. 2285, UIP 2017-05-1020, PZS-2019-02-3055 from Research Cooperability funded by the European Social Fund}	
	
	\begin{abstract}
Let $F\in\mathbb{C}[x,y]$ be a polynomial, $\gamma(z)\in \pi_1(F^{-1}(z))$ a non-trivial cycle in a generic fiber of $F$ and let $\omega$ be a polynomial $1$-form, thus defining a polynomial deformation  $dF+\epsilon\omega=0$ of the integrable foliation given by $F$. 

We study different invariants: the \emph{orbit depth} $k$, the \emph{nilpotence class} $n$, the \emph{derivative length} $d$ associated with the couple $(F,\gamma)$. These invariants bound the \emph{length $\ell$ of the first nonzero Melnikov function} of the deformation $dF+\epsilon\omega$ along $\gamma$.

We study in detail a simple example of a polynomial $F$ given as product of four lines. We show how these invariants vary depending on the relative position of the four lines and relate it also to the length of the corresponding Godbillon-Vey sequence. 
We formulate a conjecture motivated by the study of this example. 
\end{abstract}


\maketitle

	\section{Introduction, Main Results and Conjectures}
This work is motivated by the  \emph{16-th Hilbert's problem} or rather its \emph{infinitesimal version}. As it is known, the second part of 
Hilbert's 16-th  problem asks for an upper bound in function of the degree for the number of real \emph{limit cycles}, i.e. isolated periodic orbits of polynomial vector fields in the plane. The problem is far from been solved and the existence of such a number is open even for quadratic vector fields. 

Arnold formulated the \emph{infinitesimal Hilbert's problem}, which asks for a bound on the number of (real)  limit cycles that can arise under polynomial deformations from an integrable polynomial differential equation in the plane. This bound must be uniform in the sense that it must depend exclusively on the degree of the integrable one-form defining the corresponding foliation, and the degree of the polynomial that realizes the perturbation.

In one of its forms Arnold's infinitesimal Hilbert problem \cite{A} studies the following situation.

Let  $F(x,y)\in\mathbb{R}[x,y]$ be a square-free polynomial, $z$ a regular value of $F$  and $\gamma(z)\subset F^{-1}(z)$ a continuous family of real cycles of $F^{-1}(z)$. We will consider the complexification of the polynomial $F$ which, for the sake of simplicity, we will denote again by $F$.
The polynomial $F$ defines a singular fibration and hence a singular foliation given by the integrable one-form
\begin{equation}\label{init}
dF=0\,.
\end{equation}

Consider the polynomial deformations depending on the parameter $\epsilon$
\begin{equation}\label{def}
dF+\epsilon\omega=0\,.
\end{equation}

As usual, one defines the displacement function (holonomy map minus identity) $\Delta_\epsilon$ of the deformation \eqref{def}	 along $\gamma(z)$.
The isolated zeros of the displacement map $\Delta_\epsilon$ are in correspondence with the limit cycles of \eqref{def}. 

The Taylor series expansion with respect to $\epsilon$ at $0$ of  $\Delta_\epsilon$ is given by 
\begin{equation}\label{Delta}
\Delta_\epsilon(z)=\sum_{i=\mu}\epsilon^{i}M_i(z).
\end{equation}
The functions $M_i$ are called \emph{Melnikov functions}. 
We assume $M_\mu\not\equiv0$ and call it the \emph{first non-zero Melnikov function}. 
By the implicit function theorem, for regular values of $z$, it carries the main information on the number of limit cycles that can arise under the deformation. 
This motivates the study of this function $M_\mu$, the subject of the infinitesimal Hilbert 16th problem.

It is known that $M_\mu$ is an iterated integral \cite{G}.  Let $\ell$ denote  its \emph{iterated integral length}, 
or shorter, simply \emph{length}. It measures the complexity of $M_\mu$.
 Bounding this length, would be a key step towards the bound of the limit cycles arising under perturbation. In \cite{GI}, Gavrilov and Iliev gave a condition which guarantees that $M_\mu$ is an abelian integral (iterated integral of length $1$). We showed the necessity of this condition in \cite{MNOP1}.

\medskip
Let $z$ be a generic value of $F$. Consider the fundamental group $\pi_1$ of $F^{-1}(z)$, and let $L_j$ be 
its lower central sequence; \begin{equation}\label{eq:L_j}
L_{j}=[L_{j-1},\pi_1], \quad \text{where }\,\, L_1=\pi_1. 
\end{equation}
Using the fibration given by $F$, we define the monodromy group, which is a subgroup of the group of automorphisms of $\pi_1$. Then denote by $\mathcal{O}=\OO_\gamma$ the normal subgroup of $\pi_1$ generated by the  orbit of $\gamma$ under the action by the monodromy group  (see \cite{MNOP1}). 

In \cite{MNOP1}, we defined the \emph{orbit depth} $k$ of the cycle $\gamma$ of $F$ by
\begin{equation}\label{k}
k=\sup\{j\geq1| \mathcal{O}\cap L_j\not\subset [\OO,\pi_1]\}=\min\{j\geq1| \mathcal{O}\cap L_{j+1}\subset [\OO,\pi_1]\},
\end{equation}
and  showed that
the orbit depth $k$ bounds the length of the iterated integral $M_\mu$, for the displacement map $\Delta_{\epsilon}$ of any polynomial deformation \eqref{def}, i.e. that
\begin{equation}\label{lk}
\ell\leq k.
\end{equation}

In \cite{MNOP2}, we gave an example of a system having unbounded depth and formulated 
 the following conjectures:

\begin{conjecture}\label{conj} \hfill
\begin{enumerate} 
\item[(i)]
 For any polynomial  $F$ and any non-trivial cycle $\gamma$ of $F$, either the depth is unbounded, or it is $1$, or $2$. 

\item[(ii)] For any $F$ and its cycle $\gamma$, either there exist deformations $\omega$, whose first non-zero Melnikov function $M_\mu$ is of arbitrary high length, or for any deformation $\omega$, the length of $M_\omega$ is $1$ or $2$.
\end{enumerate}
\end{conjecture}
\medskip

This conjecture is similar in spirit to the result of Casale \cite{C} , concerning the first integral and the length of the corresponding Godbillon-Vey sequence of a system
or  the Tits alternative \cite{T}. 
\\

One of the central components in the study is the \emph{orbit complement abelianization group} or $\mathcal{O}$\emph{-abelianization}
\begin{equation}\label{G}
\Ocom:=\frac{\pi_1}{\mathcal{O}\cap L_2}.
\end{equation}

Note that it is associated to the couple $(F,\gamma)$ and does not depend on the deformation form $\omega$.

In order to give some steps towards the proof of the above conjecture, we consider some other, more classical invariants associated to the group $\Ocom$ : 
its \emph{nilpotence class} $n$ and its \emph{derivative length} $d$, which we relate to the orbit depth $k$ and length $\ell$ of the first nonzero Melnikov function (see section \ref{sect:groups} for definitions). 

\begin{proposition}\label{prop:ineq} Let $F\in\Bbb{R}[x,y]$ be a polynomial,  $\gamma\in\pi_1(F^{-1}(z))$ be a real cycle and $\omega$ a polynomial $1$-form as above, and let $\Ocom$ be the group defined in \eqref{G}. 
Let $n$ be the \emph{nilpotence class} of $\Ocom$, $d$ the \emph{derivative length} of $\Ocom$, $k$ the \emph{orbit depth} and $\ell$ the  \emph{length of the first nonzero Melnikov} function of the deformation \eqref{def} along $\gamma$. Then the following inequalities hold 
\begin{equation}\label{inequalities}
\ell\leq k\leq n+1, \quad d\leq n.
\end{equation} 

\end{proposition}

\begin{remark}
Although $d\le n$, there is no evident relationship between $d$ and $k$.
\end{remark}

\begin{approach} \label{GA}
The Poincar\'e (holonomy) map of a deformation \eqref{def} induces a homomorphism 
\begin{equation}\label{eq:Poincare general}
P_\omega:\pi_1\to \C[\epsilon]\otimes\operatorname{Diff}\left(\C,0\right), \quad	P_\omega(\delta)=\left\{z\mapsto z+\epsilon\int_{\delta}\omega+O(\epsilon^2)\right\},
\end{equation}
where $P_{\omega}(\delta)$ is the Poincar\'e map with respect to the foliation \eqref{def} along the cycle $\delta$.

Assume that the deformation \eqref{def} preserves the continuous families of cycles corresponding to elements of $\OO\cap L_2$. This means that $P_\omega(\OO\cap L_2)=\{Id\}$ and the map \eqref{eq:Poincare general} descends to the homomorphism 
\begin{equation}\label{eq:homomorphism}
		P_{\omega}:\Ocom\rightarrow \C[\epsilon]\otimes\Diff(\C,0).
\end{equation}
 Moreover, if one defines a deformation 
form $\omega$ in such  a way that the obtained subgroup of diffeomorphisms is parabolic, then the group $\Ocom$ inherits properties of subgroups of parabolic diffeomorphisms. 
Using these properties, one could conclude that either $\Ocom$ is abelian, and so the orbit depth is less or equal to 2, or it is non-solvable. 
\medskip 

However, the difficulty here lies in finding a one-form $\omega$ such that, for a cycle $\sigma\in [\pi_1,\pi_1]\setminus\OO\cap L_2$, $P_{\omega}(\sigma)\ne id$, but it is the identity along any commutator in the orbit. 

\end{approach}

\bigskip

Here, we realize the above approach in the case of a Hamiltonian $F$ given as a product of four real lines
\begin{equation}\label{F}
\begin{aligned}
&F(x,y)=f_1(x,y)f_2(x,y)f_3(x,y)f_4(x,y), \\
& f_i=a_ix+b_iy+c_i, (a_i,b_i)\ne (0,0),\quad (a_i,b_i)\ne(a_j,b_j), \text{ for } i\ne j.
\end{aligned}
\end{equation}
Recall that, as studied in \cite{MNOP2}, in the above case, if $f_i$, $i=1,\ldots,4,$ consist of two pairs of parallel lines, the orbit depth $k$ is infinite. 
We  distinguish three cases of a hamiltonian given by a product of four lines:

\begin{enumerate}
\item The four lines are in \emph{generic} position (no parallel lines among $f_i$, $i=1,\ldots,4$, for different indices).\\
\item one of the bounded domains bounded by these lines is a quadrilateral with exactly one pair of parallel opposite sides (we call it a \emph{trapezoid}).\\
\item one of the bounded domains bounded by these lines is a parallelogram (we call  it a \emph{parallelogram}).\\
\end{enumerate}

\begin{theorem} \label{Thm:1}
Let $F$ be the product of four real lines \eqref{F}, and  $\gamma\in \pi_1(F^{-1}(z))$ a continuous family of real cycles of $F$. 
\begin{enumerate}
\item In the case that the  four lines are in generic position, then  $\Ocom$ is abelian, so its nilpotence class $n$ and derived length $d$ are equal to 1. Hence the orbit depth $k\leq2$, and the length of the first non-zero Melnikov function is $\ell\le 2$.
\item In the trapezoid case $\Ocom$ is non-solvable $d=n=\infty$, but the orbit depth is $k\leq2$. Hence,  the length of the first non-zero Melnikov function of any deformation is bounded by $2$,  $\ell\leq2$.  
\item In the parallelogram case $\Ocom$ is non-solvable $d=n=\infty$ and the orbit depth is infinite $k=\infty$. 
\end{enumerate}
\end{theorem}

\begin{remark}
In order to prove the conjecture that the orbit depth is an optimal bound for the length of Melnikov functions \cite{MNOP1},
the problem resides of the realization of the length. In \cite{MNOP2} we provide a deformation of the parallelogram having a first non-zero Melnikov function of length 3, which already distinguishes this case from the trapezoid, in terms of the length, where the length is less than or equal to 2.
\end{remark}

 The case of four lines in generic position has been studied in \cite{Uribe,MNOP1}, and it is known that $[\pi_1,\pi_1]\subseteq\OO$. Hence, 
 $\Ocom$ is abelian, and $k\leq 2$. We recall it here for completeness. In this work, our aim is to focus on more degenerated situations, the trapezoid and parallelogram cases.

Since the nilpotence class $n$ of $\Ocom$ provides an upper bound for the orbit depth $k$, we also want to understand conditions under which 
$\Ocom$ is non-nilpotent (or non-solvable). From the definition of the homomorphism in \eqref{eq:homomorphism} we see that the nilpotence class of $\Ocom$ is greater than the nilpotence class of its image under $P_{\omega}$. 
In this sense, the class of nilpotence of $\Ocom$ is related with the type of deformations preserving centers in $\OO\cap L_2$, and therefore, with the type of integrability of these deformations.

\begin{theorem}\label{Thm:2}\hfill
	\begin{enumerate}
\item In the parallelogram and the trapezoid case, there exist polynomial $1$-forms $\omega$ such that the  deformations \eqref{def} preserve the pairs of parallel lines and have a first integral of Riccati type. 
\item Moreover, there are deformations, preserving the pairs of parallel lines,  having  Godbillon-Vey sequences of any finite length.
	\end{enumerate}
\end{theorem}

\begin{remark}
In \cite{CLNLP} the authors define the length of a foliation as the minimal length among all Godbillon-Vey sequences for the foliation. They mention that they do not know any example of finite length greater than 4. Deformations in second part of Theorem \ref{Thm:2} could provide such examples if one can prove that its Godbillon-Vey sequence has optimal length.

\end{remark}

\begin{conjecture}\hfil
	\begin{enumerate}

\item The non-solvability of the  group $\Ocom$ is characterized by the presence of a pair of parallel curves in the hamiltonian foliation.
\item The non-bounded orbit depth is characterized by the presence of two pairs of parallel curves in the hamiltonian foliation.

\item The type of singularity given by a pair of parallel curves at the line at infinity characterizes the non-solvability of $\Ocom$.\\
	\end{enumerate}

\end{conjecture}

Here by parallel curves, we mean two  level curves $f^{-1}(c_1)$ and $f^{-1}(c_2)$, of the same function $f$, for $c_1\ne c_2$.  Parallel curves interesect only at the line at infinity. 

\section{Nilpotence class and derivative length}\label{sect:groups}

\begin{definition}\hfill
	\begin{enumerate}
		\item[(i)]
		Given a group $G$, let $G_i$ be its \emph{lower central sequence}:
		
		$$
		G=G_1\supset G_2\supset \cdots, \quad G_{j+1}=[G_j,G].
		$$
		If there exists $j\in\N$ such that $G_j=\{e\}$, we say that the group is \emph{nilpotent} and define
		its \emph{nilpotence class} $n=n(G)$ as
		\begin{equation}\label{n}
		n=\min\{j\geq1| G_{j+1}= \{e\}\},
		\end{equation}
		where $e$ is the identity element in $G$.
		
		\item[(ii)] Similarly, the \emph{upper central sequence}  $G^j$ is 
		$$
		G=G^0\supset G^1\supset\cdots,\quad G^{j+1}=[G^j,G^j].
		$$
		If there exists $j$ such that $G^{j}=\{e\}$, we say that the group is \emph{solvable} and define its \emph{derived length} $d$ as
		\begin{equation}\label{n}
		d=\min\{j\geq1| G^{j}= \{e\}\}.
		\end{equation}
	\end{enumerate}
\end{definition}
Note that
\begin{equation}\label{eq:G^jG_j+1}
G_{j+1}\supset G^{j}.
\end{equation}

This gives
$d\le n$ and in particular any nilpotent group is solvable.
\medskip

We apply these two notions to the orbit complement abelianization group $\Ocom$ given in \eqref{G}.

\begin{proof}[Proof of Propostion \ref{prop:ineq}]
	 In order to prove that $k\le n+1$, it suffices to prove that 
	$
	(\Ocom)_j=\{\bar{e}\}
	$ implies $\OO\cap L_{j+1}\subset[\OO,\pi_1]$, for any $j$.

	We claim first that the assumption $(\Ocom)_j=\{\bar{e}\}$, implies
	\begin{equation}\label{LjO}
	L_j\subset\OO.
	\end{equation} 
	Indeed, let
	$\sigma=[[\cdots[\gamma_1,\gamma_2],\cdots,\gamma_{j-1}],\gamma_{j}]$ be a generator of $L_j$ and let $\hat{\sigma}\in\Ocom$ be the  class of this $\sigma_i$ in $(\Ocom) $. Then $\hat{\sigma}\in\left(\Ocom\right)_j=\{e\}$, i.e. $\sigma\in\OO\cap L_2$. We thus have \eqref{LjO}.
	
	Hence, $L_{j+1}=[L_j,\pi_1]\subset[\OO,\pi_1]$, showing that $\OO\cap L_{j+1}\subset[\OO,\pi_1]$.
	
 	The relation $d\le n$ comes from expression \eqref{eq:G^jG_j+1}. On the other hand, it is known from \cite{MNOP1} that the length $\ell$ of the first non-zero Melnikov function is bounded by the orbit depth $k$. Therefore, $\ell\le k\le n+1$.
\end{proof}

\section{Germs of diffeomorphisms}

Given a germ of diffeomorphism $f\in \Diff(\C,0)$, we say that it is \emph{parabolic} if it is of the form
$f(z)=z+o(z)$. If $f$ is not the identity, then $f=z+az^{p+1}+o(z^{p+1})$, with $a\neq 0$. We call  $p$ the \emph{level} of $f$.
Let $\Diff_1(\C,0)\subset \Diff(\C,0)$ denote the subgroup of parabolic germs.

The General approach \ref{GA}  is based on well-known facts about the solvability of the group of parabolic germs $\Diff_1(\C,0)$. 

\begin{lemma} \label{lem:com}(Proposition 6.11, \cite{IY})\label{lemma:levels}
	Let $f=z+az^{p+1}+\cdots$ and $g=z+bz^{q+1}+\cdots$. Then $[f,g](z)=z+ab(p-q)z^{p+q+1}+o(z^{p+q+1})$.
\end{lemma}

\begin{proposition} (Theorem. 6.10 \cite{IY})\label{prop:Gsolvable}
	Let $G$ be a finitely generated subgroup of parabolic germs $\Diff_1(\C,0)$. Then (i) or (ii) holds
	\begin{enumerate}
		\item[(i)]
		$G$ is abelian, i.e. of nilpotence class $n(G)\leq1$. 
		\item[(ii)]
		$G$ is not solvable (hence not nilpotent).
	\end{enumerate}
\end{proposition}

\begin{lemma}\label{lemma:GDiff_1}\cite{IY}
	Let $G$ be a finitely generated subgroup of parabolic germs $\Diff_1(\C,0)$. Then $G$ is solvable if and only if it is abelian. Moreover, one of the  following statements  holds:
	\begin{enumerate}	
		\item[(i)]
		$G$ is not abelian and there exist two diffeomorphisms $f$ and $g$ in $G$ of different levels.
		\item[(ii)]
		$G$ is abeliant  and all diffeomorphisms in $G$ are of the same level. 
	\end{enumerate}
\end{lemma}
\begin{proof}
	Suppose  $G\neq \{id\}$.
	We now consider two cases: either there exist two germs $f=z+az^{p+1}+o(z^{p+1})$ and $g=z+bz^{q+1}+o(z^{q+1})$ in $G$ of different level $q\neq p$, or all germs are of the same level. 
	
	In the first case, by Lemma \ref{lemma:levels},  $h=[f,g](z)=z+ab(p-q)z^{p+q+1}+o(z^{p+q+1})\neq z$. Therefore, $G_1=[G,G]\neq \{id\}$. Applying inductively the same reasoning on $h$ and $g$, one can show that $G_{\ell}\neq \{id\}$ and $G^\ell\not=\{id\}$, for all $\ell$. Thus, $G$ is not nilpotent, nor solvable. 
	
	Now, suppose that all the elements of $G$ have the same level $p$. Then, by Lemma \ref{lemma:levels}, given $f=z+az^{p+1}+o(z^{p+1})$ and $g=z+bz^{p+1}+o(z^{p+1})$, the element $[f,g]$ is either of a level strictly greater than $p$ or the identity. The first option is impossible by the assumption, so 	 it follows that $G_2=G^1=\{id\}$. Thus $G$ is abelian. 
\end{proof}

\section{Proof of Theorem \ref{Thm:1}}
\begin{proof}
In the case (1), of product of lines in generic position, it is known that $\Ocom$ is abelian \cite{Uribe}. Hence, the orbit depth of the real cycle, as well as the length of the first non-zero Melnikov function for any deformation, is bounded by 2, by \cite{MNOP1}.\\

To prove that $\Ocom$ is non-solvable for cases (2) and (3) we follow the same strategy. In both cases, by an affine change of coordinates we can assume that the hamiltonian is given by
$F=(x-1)(x+1)f_3f_4$, where $f_3$ and $f_4$ are linear factors (non-parallel in case (2) and parallel in case (3)).

Now, consider the foliation $\mathcal{F}_\omega=\{dF+\epsilon \omega=0\}$, where $\omega=F^2(\frac{dx}{x-1}+F\frac{dx}{x+1})$, and the homomorphims defined by the holonomy with respect to $\mathcal{F}_\omega$:
$$P_{\omega}:\Ocom\rightarrow \C[\epsilon]\otimes Diff^1(\C,0).$$

We define the subgroup $G:=P_{\omega}(\Ocom)$ of the group of parabolic difeomorphims with coefficients depending on $\epsilon$. Note that this morphism $P_{\omega}$ is well defined, since $\mathcal{F}_\epsilon$ preserves the center at the origin, because it has a reflexion symmetry with respect to $y$-axis. And, since $P_{\omega}[\delta_1,\delta_2]=z+\epsilon^2W(z^2,z^3)\int_{[\delta_1,\delta_2]}\frac{dx}{x-1}\frac{dx}{x+1}+O(\epsilon^3)\neq z$, where $W$ is the Wronskian, then $G$ is not Abelian, and therefore non-solvable. 

In \cite{MNOP2} it is proved that the parallelogram has unbounded orbit depth. Now we show that for the trapezoid the orbit depth is $k\le 2$. By an affine change of coordinates let

	\begin{equation}\label{eq:FF}
	F=(x-1)(x+1)f_3f_4,
	\end{equation}
	with the configuration of Figure \ref{Fig:Trap}.
	The zero level $F^{-1}(0)$ consists of four lines, enclosing a quadrilateral and a triangle. Let $\gamma$ be the real cycle in $F^{-1}(z')$, $z'>0$, close to zero enclosed by the quadrilateral and let $\gamma_1$ be the real cycle in $F^{-1}(z'')$, $z''<0$, close to zero enclosed by the triangle, both with positive orientation.
	Let $p_i$ be the vertices of the quadrilateral oriented positively with $p_2$ and $p_3$ the commun vertices with the triangle and let $p_5$ be the last vertex of the triangle. 
	Let 
	$\delta_i$, be the vanishing cycles at the  vertices $p_i$, $i=1,\ldots,5$
	taken so  that the intersection numbers verify $(\gamma,\delta_i)=1$, $i=1,\ldots,4$ and $(\gamma_1,\delta_5)=1$.

\begin{figure}[h]
	\begin{center}
		\includegraphics[height=5cm]{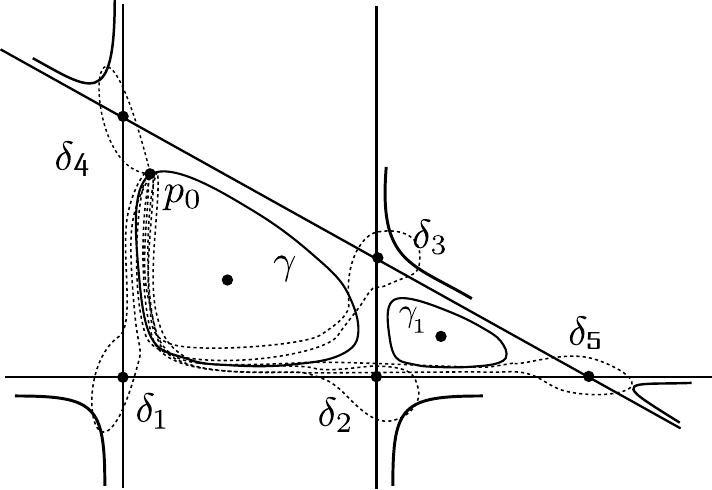}.
	\end{center}
	\caption{The real loops $\gamma(z)$ and $\gamma_1(z)$ and the complex vanishing loops $\delta_i(z)$ as elements of $\pi_1\left(F^{-1}(z),p_0\right)$.}
	\label{Fig:Trap}
\end{figure}


Using the Gauss-Manin connection, we continue analytically $\gamma, \gamma_1, \delta_i$ to a non-singular curve $F^{-1}(z)$, $0<iz\ll 1$, along segments $[z', z]$ and $[z'',z]$.
	Recall that, by \cite{AC}, the intersection numbers verify:
	\begin{equation}\label{eq:gammagamma1}(\gamma,\gamma_1)=1
	\end{equation}
	and
	\begin{equation}\label{eq:F}
	(\gamma_1,\delta_2)=(\gamma_1,\delta_3)=1.
	\end{equation}
\begin{remark}
 Theorem \ref{Thm:1} is also true in the complex case if conditions \eqref{eq:gammagamma1} and \eqref{eq:F} are fulfilled.

\end{remark}	
	In 
	\cite{MNOP1}, we calculated the orbit of the real cycle for a first integral of triangle type and in
	\cite{MNOP2}, for the quadrilateral type. The quadrilateral case in a neighborhood of the quadrilateral is the same as in the parallelogram case. Now, due to \eqref{eq:gammagamma1}, in our case \eqref{eq:FF}, the orbit of $\gamma$ is generated by the union of the orbits in these two cases. 
	This gives:
	$$
	\begin{aligned}
	&\frac{\mathcal{O}_\gamma}{[\OO_\gamma,\pi_1]}=\\
	&=<\{\gamma,\delta_1\delta_2\delta_3\delta_4,[\delta_1\delta_2,\delta_2\delta_3],[\delta_1\delta_2,[\delta_2,\delta_2\delta_3]],\ldots,
	[\delta_1\delta_2,[\delta_2,\ldots[\delta_2,\delta_2\delta_3]],\ldots\}\cup\mathcal{O}_{\gamma_1}>.
	\end{aligned}
	$$
	We know that $[\delta_2,\delta_3]$ belongs to the orbit of $\gamma_1$ and hence to $\mathcal{O}_\gamma$. Hence, $[\delta_2,\delta_2\delta_3]=\delta_2[\delta_2,\delta_3]\delta_2^{-1}$ belongs to the orbit as well. This gives that in the above expression for the orbit $\mathcal{O}_\gamma$ the commutator $[\delta_1\delta_2,[\delta_2,\delta_2\delta_3]]$ belongs to  $[\mathcal{O},\pi_1]$. The same is hence true for all terms which contain it (all the terms following it in the above braces). 
	We know \cite{MNOP1} that the depth of $\mathcal{O}_{\gamma_1}$ is two. 
	All the terms following the term $[\delta_1\delta_2,\delta_2\delta_3]$ are in $[\OO,\pi_1]$, showing that the depth $k$ verifies $k\leq2$ and hence by \cite{MNOP1} the length of the first non-zero function $M_\mu$ of any deformation $\omega$ is $\ell\leq2$.

\end{proof}

\section{Types of integrability of the deformations}
We will study the Godbillon-Vey sequence for the foliation \begin{equation}\label{eq:foliation}
dF+\epsilon\omega=0,
\end{equation} with $F=f_1f_2f_3f_4$ and $\omega=p_1(F)\frac{df_1}{f_1}+p_2(F)\frac{df_2}{f_2}$, where $f_1=x-1$, $f_2=x+1$, $p_1, p_2$ are polynomials in $F$, and  $f_3$, $f_4$ are linear factors different from $f_1$ and $f_2$. For the parallelogram, $f_3$ and $f_4$ define parallel lines, while for the trapezoid they do not.
By explicit computation of Godbillon-Vey sequences we will show that the foliation $dF+\epsilon(\frac{df_1}{f_1}+F\frac{df_2}{f_2})=0$ is Liouville integrable, while $dF+\epsilon(F\frac{df_1}{f_1}+F^2\frac{df_2}{f_2})=0$ is Riccati integrable. Moreover, for $\omega=p_1(F)\frac{df_1}{f_1}+p_2(F)\frac{df_2}{f_2}$, with $n=\deg\{\deg p_1,\deg p_2\}$, foliation \eqref{eq:foliation} admits a Godbillon-Vey sequence of length $\le n$, which increases complexity of the first integral for those cases. Therefore, for $n\ge 3$ the foliation \eqref{eq:foliation} could have a first integral which is not of Riccati type, with infinite dimensional pseudo-group of holonomy.\\

We recall that a Gobdillon-Vey sequence for a meromophic 1-form $\eta_0$ is a sequence $\{\eta_0,\eta_1,...\}$, such that \begin{equation}\label{eq:GV}
\begin{aligned}
d\eta_0&=\eta_0\wedge\eta_1\\
d\eta_1&=\eta_0\wedge\eta_2\\
d\eta_2&=\eta_0\wedge\eta_3+\eta_1\wedge\eta_2\\
d\eta_3&=\eta_0\wedge\eta_4+2\eta_1\wedge\eta_3\\
\vdots&\\
d\eta_n&=\eta_0\wedge\eta_{n+1}+\sum_{k=1}^{n}\binom{n}{k}\eta_k\wedge\eta_{n-k+1}\\
\end{aligned}
\end{equation}
It is of length $\ell$ if $\eta_j=0$ for all $j\ge \ell$.

It is known that $\eta_0$ has a first integral of Liouville type, resp. of Riccati type, if and only if, it admits a Godbillon-Vey sequence of length 2, resp. of length 3.

\subsection{Degree on $F$ equals 1}
Let us start analyzing the case $\omega=F\frac{df_1}{f_1}$. We denote $\eta_0:=dF+\epsilon F\frac{df_1}{f_1}$. Then $$d\eta_0=\epsilon dF\wedge \frac{df_1}{f_1}.$$
Thus, we can take $\eta_1:=\epsilon\frac{df_1}{f_1}=\frac{df_1^\epsilon}{f_1^\epsilon}$. Then, $d\eta_1=0$, which means $\eta_2=0$. Hence, $\eta_0$  admits a Godbillon-Vey sequence of length 2, and therefore \eqref{eq:foliation} is Liouville integrable. Moreover, Casale's article \cite{C} provides a way to compute from this sequence the first integral. This is given by $dH=G\eta_0$, where $\eta_1=\frac{dG}{G}$. In this case, we have $G=f_1^\epsilon$, then the first integral $H$ is given by $$f_1^\epsilon\eta_0=f_1^\epsilon(dF+\frac{df_1^\epsilon}{f_1^\epsilon})=f_1^\epsilon dF+Fdf_1^\epsilon=d(f_1^\epsilon F),$$
thus $H=f_1^\epsilon F$, which lies in a Liouvillian extension.
\\

Now we analyze the following case $\eta_0:=dF+\epsilon (\frac{df_1}{f_1}+F\frac{df_2}{f_2})$, which will give us a way to proceed for higher degree. Substituting $f_1=x-1$, $f_2=x+1$, we obtain $\eta_0=dF+\epsilon(\frac{1}{x-1}+\frac{F}{x+1})dx$. Denote $\varphi(x,F)=\frac{1}{x-1}+\frac{F}{x+1}$. Then $\eta_0=dF+\epsilon\varphi dx$, and \begin{equation}\label{eq:deta_0}
d\eta_0=\epsilon d\varphi\wedge dx.
\end{equation}
Note that $d\varphi=\varphi_FdF+\varphi_xdx$, where $\varphi_F=\frac{\partial\varphi}{\partial F}(x,F)=\frac{1}{x+1}$, and $\varphi_x=\frac{\partial \varphi}{\partial x}(x,F)=\frac{-1}{(x-1)^2}-\frac{F}{(x+1)^2}$.
Then, \eqref{eq:deta_0} is equal to \begin{equation}
d\eta_0=\epsilon \varphi_FdF\wedge dx.
\end{equation}

Define $\eta_1:=\epsilon\varphi_Fdx$. It satisfies the equation $d\eta_0=\eta_0\wedge \eta_1=(dF+\epsilon\varphi dx)\wedge \epsilon\varphi_Fdx$. On the other hand, $d\eta_1=\epsilon d\varphi_F\wedge dx$, where $d\varphi_F=\varphi_{FF}dF+\varphi_{Fx}dx$. That is, $d\eta_1=\epsilon\varphi_{FF}dF\wedge dx$, which is zero because $\varphi_{FF}=0$. Then, we can take $\eta_2=0$. Thus, $\eta_0$ admits a Godbillon-Vey sequence of length $2$, and therefore is Liouville integrable.
To compute the first integral $H$, we have to solve the equations $dH=G\eta_0$, and $\eta_1=\frac{dG}{G}$. For this, we note that $\eta_1=\epsilon\frac{dx}{x+1}=\frac{df_2^\epsilon}{f^\epsilon_2}$. Then, $dH=f_2^\epsilon\eta_0=f_2^\epsilon(dF+\epsilon (\frac{df_1}{f_1}+F\frac{df_2}{f_2}))$, which corresponds to $$f_2^\epsilon (dF+\frac{df_1^\epsilon}{f_1^\epsilon}+F\frac{df_2^\epsilon}{f_2^\epsilon})=f_2^\epsilon dF+f_2^\epsilon\frac{df_1^\epsilon}{f_1^\epsilon}+Fdf_2^\epsilon=d(f_2^\epsilon F+\int f_2^\epsilon\frac{df_1^\epsilon}{f_1^\epsilon} ).$$ Therefore $H=f_2^\epsilon F+\int f_2^\epsilon\frac{df_1^\epsilon}{f_1^\epsilon}$. Substituting $f_1$, $f_2$ we get $\int f_2^\epsilon\frac{df_1^\epsilon}{f_1^\epsilon}=\int \epsilon(x+1)^\epsilon\frac{dx}{x-1}$.

Note that the above construction works more generally  if $f_1$ and $f_2$ 
correspond to two parallel curves.

\section{Proof of Theorem \ref{Thm:2}} \begin{proof}

We consider the deformation $dF+\epsilon (F\frac{df_1}{f_1}+F^2\frac{df_2}{f_2})$, where $f_1=x-1$, $f_2=x+1$, and define \begin{equation}
\eta_0:=dF+\epsilon\left(\frac{F}{x-1}+\frac{F^2}{x+1}\right)dx.
\end{equation} 

Denote $\varphi(x,F)=\frac{F}{x-1}+\frac{F^2}{x+1}$. Then $\eta_0=dF+\epsilon\varphi dx$, and since $d\varphi=\varphi_FdF+\varphi_xdx$, where 
$\varphi_F$ is the partial derivative of $\varphi$ with respect to $F$:
$\varphi_F=\frac{1}{x-1}+\frac{2F}{x+1}$, we have \begin{equation}
d\eta_0=\epsilon \varphi_FdF\wedge dx.
\end{equation}
Thus, we define $\eta_1:=\epsilon\varphi_Fdx$. It satisfies the equation $d\eta_0=\eta_0\wedge\eta_1$.
Now we consider $d\eta_1=\epsilon d\varphi_F\wedge dx$. Then, again, writing $d\varphi_F=\varphi_{FF}dF+\varphi_{Fx}dx$, we have $d\eta_1=\epsilon\varphi_{FF}dF\wedge dx$.
So, we can take $\eta_2:=\epsilon\varphi_{FF}dx$. It satisfies the equation $d\eta_1=\eta_0\wedge\eta_2$. Continuing with the Godbillon-Vey sequence we consider $d\eta_2=\epsilon d\varphi_{FF}\wedge dx$. But, since the degree of $\varphi$ in $F$ is 2, $\varphi_{FF}$ depends only on $x$, therefore $d\varphi_{FF}\wedge dx=0$. Now, for the equation $d\eta_2=\eta_0\wedge\eta_3+\eta_1\wedge\eta_2$, since $\eta_1\wedge\eta_2=0$, we can take $\eta_3=0$. Hence, $\eta_0$ admits a Godbillon-Vey sequence of length 3, and therefore it has a first integral of Riccati type.

Using Casale's method \cite{C} we can compute the first integral by solving the following equations:

\begin{equation}
\begin{aligned}
dH&=G_1\eta_0\\
dG_1&=G_1(\eta_1+\frac{2}{G_2}\eta_0)\\
dG_2&=\frac{G_2^2}{2}\eta_2+G_1\eta_1+\eta_0,
\end{aligned}
\end{equation}
with $\eta_0=dF+\epsilon(\frac{F}{x-1}+\frac{F^2}{x+1})dx$, $\eta_1=\epsilon(\frac{1}{x-1}+\frac{2F}{x+1})dx$, and $\eta_2=\epsilon(\frac{2}{x+1})dx$.\\

 Now, we analyze what happens when we increase the degree in $F$ for the foliation \eqref{eq:foliation}.
Let \begin{equation}
\eta_0:=dF+\epsilon \left(F^2\frac{df_1}{f_1}+F^3\frac{df_2}{f_2}\right).
\end{equation} Substituting $f_1=x-1$, $f_2=x+1$, we obtain $\eta_0=dF+\epsilon(\frac{F^2}{x-1}+\frac{F^3}{x+1})dx$. Following the same process as before, we write $\eta_0=dF+\epsilon\varphi dx$, where $\varphi(x,F)=\frac{F^2}{x-1}+\frac{F^3}{x+1}$. Since $d\varphi=\varphi_FdF+\varphi_xdx$, we have \begin{equation}
d\eta_0=\epsilon \varphi_FdF\wedge dx,
\end{equation}
where $\varphi_F=\frac{2F}{x-1}+\frac{3F^2}{x+1}$.

Thus, we define $\eta_1:=\epsilon\varphi_Fdx$, note that $\deg_F\varphi_F=2$. It satisfies the equation $d\eta_0=\eta_0\wedge\eta_1$.
Now we consider $d\eta_1=\epsilon d\varphi_F\wedge dx$. Writing $d\eta_F=\varphi_{FF}dF+\varphi_{Fx}dx$, we have $d\eta_1=\epsilon\varphi_{FF}dF\wedge dx$.
So, we can take again $\eta_2:=\epsilon\varphi_{FF}dx$, where now $\deg_F\varphi_{FF}=1$. It satisfies the equation $d\eta_1=\eta_0\wedge\eta_2$. Then $d\eta_2=\epsilon d\varphi_{FF}\wedge dx$. We want $\eta_3$ such that $d\eta_2=\eta_0\wedge\eta_3+\eta_1\wedge\eta_2$. Since $\eta_1\wedge\eta_2=0$, we can take $\eta_3=\epsilon\varphi_{FFF}dx$. Consider $d\eta_3=\epsilon d\varphi_{FFF}\wedge dx$. This is zero, since $\varphi_{FFF}$ depends only on $x$. We want now $\eta_4$ verifying the equation $d\eta_3=\eta_0\wedge\eta_4+2\eta_2\wedge\eta_3$. But, since $\eta_2\wedge\eta_3=0$, we can take $\eta_4=0$. Moreover, from this construction $\eta_k\wedge\eta_\ell=0$ for all $k,\ell\ge 1$, hence this Godbillon-Vey sequence has length 4. Therefore, $\eta_0$ could have a first integral which is not of Riccati type. With the same process we can obtain a deformation with a Godbillon-Vey sequence of any finite length $n$, by taking the degree in $F$ of the deformation equal to $n-1$.

Note that the above computations do not depend on the factors $f_3$ and $f_4$, therefore they work for the parallelogram, as well as for the trapezoid.
\end{proof}
\begin{remark}
	From these cases one can observe that the length of the Godbillon-Vey sequence for the foliation \eqref{eq:foliation} increases with the degree in $F$ of the function $\varphi(x,F)$. So, increasing the degree in $F$ for the deformation, we can have a first integral of more complicated type.
	
	In particular, for $n=\deg_F\varphi>4$ one should obtain a Godbillon-Vey sequence of finite length $n$ higher than 4. In the paper \cite{CLNLP} (p. 25) the authors say they do not know any example of finite length greater than $4$. However, the length is defined as the minimal length among all the Godbillon-Vey sequences that the foliation can admit. So, in order to prove that these examples for $n>4$ have Godbillon-Vey length higher than $4$, one should prove that the constructed Godbillon-Vey is optimal. That is that there does not exist some other Godbillon-Vey sequence of smaller length.
\end{remark}

\end{document}